\documentclass[11pt,reqno]{amsart}
\usepackage[margin=0.7in]{geometry} 
\usepackage{amssymb,amsfonts,amsmath,bbm,mathrsfs,stmaryrd}
\usepackage{xcolor}
\usepackage{url}

\usepackage[all]{xy}

\usepackage{a4wide}
\usepackage{dsfont} % used in the command \I
\usepackage{graphicx}

\usepackage[colorlinks,
             linkcolor=blue,
             citecolor=black!75!red,
             pdfproducer={pdfLaTeX},
             pdfpagemode=None,
             bookmarksopen=true
             bookmarksnumbered=true]{hyperref}

\usepackage{tikz}
\usetikzlibrary{arrows,calc,decorations.pathreplacing,decorations.markings,intersections,shapes.geometric,through,fit,shapes.symbols,positioning,decorations.pathmorphing}

\usepackage{cleveref}

\usepackage{amsmath,amsthm,stmaryrd}
\usepackage{cleveref}
\newtheorem{lemma}{Lemma}[section]
\newtheorem{theorem}[lemma]{Theorem}
\newtheorem{proposition}[lemma]{Proposition}
\newtheorem{corollary}[lemma]{Corollary}

\theoremstyle{definition} 

\newtheorem{definitionnodiamond}[lemma]{Definition}
\newtheorem{examplenodiamond}[lemma]{Example}
\newtheorem{remarknodiamond}[lemma]{Remark}

\numberwithin{equation}{section}

\crefname{section}{Section}{Sections}
\crefformat{section}{#2Section~#1#3} 
\Crefformat{section}{#2Section~#1#3} 

\crefname{subsection}{}{Subsections}
\crefformat{subsection}{\S#2#1#3} 
\Crefformat{subsection}{\S#2#1#3}

\crefname{definition}{Definition}{Definitions}
\crefformat{definition}{#2Definition~#1#3} 
\Crefformat{definition}{#2Definition~#1#3} 

\crefname{example}{Example}{Examples}
\crefformat{example}{#2Example~#1#3} 
\Crefformat{example}{#2Example~#1#3} 

\crefname{examplenodiamond}{Example}{Examples}
\crefformat{examplenodiamond}{#2Example~#1#3} 
\Crefformat{examplenodiamond}{#2Example~#1#3} 

\crefname{remark}{Remark}{Remarks}
\crefformat{remark}{#2Remark~#1#3} 
\Crefformat{remark}{#2Remark~#1#3} 

\crefname{remarknodiamond}{Remark}{Remarks}
\crefformat{remarknodiamond}{#2Remark~#1#3} 
\Crefformat{remarknodiamond}{#2Remark~#1#3} 

\crefname{convention}{Convention}{Conventions}
\crefformat{convention}{#2Convention~#1#3} 
\Crefformat{convention}{#2Convention~#1#3} 

\crefname{lemma}{Lemma}{Lemmas}
\crefformat{lemma}{#2Lemma~#1#3} 
\Crefformat{lemma}{#2Lemma~#1#3} 

\crefname{proposition}{Proposition}{Propositions}
\crefformat{proposition}{#2Proposition~#1#3} 
\Crefformat{proposition}{#2Proposition~#1#3} 

\crefname{corollary}{Corollary}{Corollaries}
\crefformat{corollary}{#2Corollary~#1#3} 
\Crefformat{corollary}{#2Corollary~#1#3} 

\crefname{theorem}{Theorem}{Theorems}
\crefformat{theorem}{#2Theorem~#1#3} 
\Crefformat{theorem}{#2Theorem~#1#3} 

\crefname{assumption}{Assumption}{Assumptions}
\crefformat{assumption}{#2Assumption~#1#3} 
\Crefformat{assumption}{#2Assumption~#1#3} 

\crefname{equation}{}{}
\crefformat{equation}{(#2#1#3)} 
\Crefformat{equation}{(#2#1#3)}

\crefname{align}{}{}
\crefformat{align}{(#2#1#3)} 
\Crefformat{align}{(#2#1#3)}

\crefname{proofstep}{Step}{Steps}
\crefformat{proofstep}{#2Step~#1#3} 
\Crefformat{proofstep}{#2Step~#1#3}

\def\GL{{\rm GL}}
\def\Kdim{{\rm Kdim}}
\def\Proj{{\sf Proj}}
\def\Qcoh{{\sf Qcoh}}

\def\PP{{\mathbb P}}

\newcommand\cL{\mathcal L}

\newcommand\cO{\mathcal O}

\newcommand\fp{\mathfrak p}

\newcommand{\bwedge}{\mbox{\Large $\wedge$}}

\def\a{\alpha}

\def\g{\gamma}
\def\l{\lambda}

\def\s{\sigma}

\numberwithin{equation}{section}

\begin{document}

\author{Alex Chirvasitu, S. Paul Smith and Michaela Vancliff}

\address{Department of Mathematics, University at Buffalo, Buffalo, NY 14260-2900, USA}
\email{achirvas@buffalo.edu}
\address{Department of Mathematics, Box 354350, Univ.\ of Washington, Seattle, WA 98195,
USA.}
\email{smith@math.washington.edu} 
\address{Department of Mathematics, Box 19408, Univ.\ of Texas at Arlington, 
Arlington, TX 76019-0408, USA.}
\email{vancliff@uta.edu}



\begin{abstract} 
Let $R$ denote a 6-dimensional subspace of the ring $M_4(\Bbbk)$ of $4 \times 4$ matrices 
over an algebraically closed field~$\Bbbk$. Fix a vector space isomorphism $M_4(\Bbbk) \cong \Bbbk^4 \otimes \Bbbk^4$.
We associate to~$R$ a closed subscheme~${\mathbf X}_R$ of the Grassmannian of 
2-dimensional subspaces of~$\Bbbk^4$, where the reduced subscheme of ${\mathbf X}_R$ 
is the set of 2-dimensional subspaces $Q \subseteq \Bbbk^4$ such that 
$(Q \otimes \Bbbk^4) \cap R \ne \{ 0\}$.  Our main result is that if 
${\mathbf X}_R$ has minimal dimension (namely, one), then its degree is 20 when it is viewed as a subscheme of~$\PP^5$ via the Pl\"ucker embedding.

We present several examples of $\mathbf X_R$ that illustrate the wide range of 
possibilities for it; there are reduced and non-reduced examples. 
Two examples involve 
elliptic curves: in one case, ${\mathbf X}_R$ is a $\PP^1$-bundle over an elliptic curve, the second symmetric power of the curve; 
in the other, it is a curve having seven irreducible components, three of which are quartic 
elliptic space curves, and four of which are smooth plane conics. These two examples arise naturally from a problem having its roots in 
quantum statistical mechanics.

The scheme~$\mathbf X_R$ appears in non-commutative 
algebraic geometry: under appropriate hypotheses, it is isomorphic to the 
line scheme $\cL$ of a certain graded algebra determined by~$R$. 
In that context, it has been an open question for 
several years to describe such~$\cL$ of minimal dimension, i.e., 
those~$\cL$ of dimension one.  Our main result implies 
that if $\dim(\cL) = 1$, then, as a subscheme of~$\PP^5$ under the 
Pl\"ucker embedding, $\deg(\cL) = 20$.
\end{abstract}
\title[$6$-dimensional subspaces of $4\times 4$ matrices]{A geometric invariant of \\  $6$-dimensional subspaces of $4\times 4$ matrices}

\maketitle

%%%%%%%%%%%%%%%%%%%%%%%%%%%%%%%%%%%%%%%%%%%%%%%%%%%%%%%%%%%%%%%%%%%%%%%%%%%%%%%%%%%%%%%%%%%%%%%%%%%%%%%%%%%%%%%%%%%%%%%%%%%%%%%%%%%%%%%%%%%%%%%%%%%%%%%%%%%%%%%%%%%%%%%%%%%%%%%%%%%%%%%%%%%%%%%%%%%%%%%%%%%%%%%%%%%%%%%%%%%%%%%%%%%%%%%%%%%%%%%%%%%%%%%%%%%%

\section{Introduction}

\subsection{}
Our main result,  \cref{thm.main} below, requires quite a bit of notation so we begin with an approximate version 
of it that will allow us to discuss some of its applications and interpretations in this introduction.

\begin{theorem}
[Approximate version of \cref{thm.main}]
Let $V$ and $V'$ be 4-dimensional vector spaces over an algebraically closed field~$\Bbbk$ and let $R$~be a 6-dimensional subspace of $V \otimes V'$. 
Let ${\mathbf X}_R$~be the closed subscheme of the Grassmannian ${\sf G}(2,V)$ of 2-dimensional subspaces of~$V$ whose reduced locus, 
$({\mathbf X}_R)_{\rm red}$, is the set of 2-dimensional subspaces $Q \subseteq V$ such that $(Q \otimes V') \cap R \ne \{0\}$.  
Then every irreducible component of ${\mathbf X}_R$ has dimension $\ge 1$ and if $\dim({\mathbf X}_R)=1$, then 
the degree of ${\mathbf X}_R$ is 20 when it is viewed as a subscheme of the projective space $\PP^5$ via the Pl\"ucker embedding.
\end{theorem}

\subsection{Linear algebra problems}
\label{sect.lin.alg.qus}
Fix a positive integer $d$. Let $R$ be a $d$-dimensional subspace of a tensor product $V_1 \otimes \cdots \otimes V_n$ of finite-dimensional vector spaces. 
What invariants can one attach to~$R$ to distinguish one such~$R$ from another? 
We do not address this very big problem but the special case in this paper 
suggests that geometric invariants of the kind we consider might be rich and computable invariants. Our cursory examination of the literature suggests 
these are new invariants. A more thorough examination of them is likely to be fruitful and interesting.

\subsection{Subspaces of matrices}
In the setting of our main result, consider the case $V'=V^*$. Since $V \otimes V^*$ is naturally isomorphic to the space of 
linear transformations $V \to V$, i.e., to the matrix algebra $M_4(\Bbbk)$, our main result applies to 6-dimensional subspaces of $M_4(\Bbbk)$. Of course,
our result addresses only a special case of a large set of questions about subspaces of matrices. The remarks in \S\ref{sect.lin.alg.qus} apply here too.

\subsection{Invariants of non-commutative algebras}
\label{sect.ATV}
Associative algebras are often presented in terms of generators and relations. Frequently, such a presentation is unenlightening.
For example, it can be very difficult to decide if one such algebra is isomorphic to another. Furthermore, such a presentation is often a subtle invitation to calculate,
and calculate some more, without understanding the fundamental reason why 
various results are true. 

In order to focus the discussion, 
consider the case of an algebra $A=TV/(R)$ where $TV$~denotes the tensor algebra 
on a vector space~$V$ and $R$~denotes a subspace of $V \otimes V$. One can 
present~$A$ by choosing a basis for~$V$ and a basis for~$R$ that is written in 
terms of the basis for~$V$. Following an idea due to Artin-Tate-Van den Bergh 
(see \cite{ATV1,ATV2}) it is often very useful to view elements of~$R$ as 
bi-homogeneous forms on the projective variety $\PP(V^*) \times \PP(V^*)$ and to 
then consider the scheme-theoretic zero locus, $\Gamma \subseteq \PP(V^*) \times \PP(V^*)$, of~$R$. 

For the quadratic algebras considered in \cite{ATV1,ATV2} (now called 
%3-dimensional 
quadratic Artin-Schelter regular algebras of global dimension three, which 
have the property that $\dim(V)=\dim(R)=3$), 
$\Gamma$~determines $R$ in the sense that $R$~is the largest subspace 
of $V \otimes V$ vanishing on $\Gamma$.  In this way, the relations that 
define~$A$ have geometric meaning. Furthermore, in those examples, $\Gamma$~is 
either the graph of an automorphism of $\PP(V^*) = \PP^2$, or the graph of an 
automorphism of a cubic divisor in  $\PP(V^*)=\PP^2$. This interpretation of the 
relations for~$A$ is, in some sense, the key to unlocking the fine structure of 
the algebras considered in \cite{ATV1,ATV2}.

\subsubsection{}
\label{sect.YR}
By the results in \cite{ATV1,ATV2}, the scheme $\Gamma$ has a description like 
that of~${\bf X}_R$. Let $V$~be a 3-dimensional vector space and $R$~a 3-dimensional subspace of $V \otimes V$. Let ${\bf Y}_R$ be the closed subscheme of 
${\sf G}(2,V) \cong \PP^2$ consisting of the 2-dimensional subspaces 
$Q \subseteq V$ such that $(Q \otimes V) \cap R \ne \{0\}$. Every irreducible 
component of ${\bf Y}_R$ has dimension $\ge 1$ and if $\dim({\bf Y}_R)=1$,
then $\deg({\bf Y}_R)=3$. Moreover, if $TV/(R)$ is Artin-Schelter regular, then
${\bf Y}_R \cong \Gamma$. 

\subsubsection{}
If  $R=\{x \otimes y-y \otimes x \; | \; x,y \in V\}$, i.e., if $R$ is the space of skew-symmetric tensors in $V \otimes V$, then ${\mathbf X}_R={\sf G}(2,V)$; in this case
$TV/(R)$ is a polynomial ring. More generally, if $\theta \in \GL(V)$ and $R=\{x \otimes y-y \otimes \theta(x) \; | \; x,y \in V\}$, 
then ${\mathbf X}_R={\sf G}(2,V)$.

\subsubsection{}
In the thirty years since the appearance of \cite{ATV1,ATV2}, the idea of 
associating $\Gamma$ to $A$  has played a central role in the 
study of algebras of the form $TV/(R)$ and in the study of the non-commutative 
algebraic varieties having such algebras as homogeneous coordinate rings.

\subsection{Non-commutative algebraic geometry}
We became interested in the question that our main theorem answers as a result of our interest in non-commutative algebraic geometry (NCAG).
We will now say a few words about that although, in the interests of simplicity,
we will often shade the truth.

One theme in NCAG is the examination of non-commutative analogues of the projective spaces $\PP^n$. Given such a non-commutative variety, 
which we will denote by $\PP_{nc}$ for now, a basic and fruitful question has been the determination of the ``points''  and ``lines'' in  $\PP_{nc}$, and their incidence 
relations. The classification of points and lines in $\PP_{nc}$ can be formulated as moduli problems and there are moduli spaces for each 
that are called the {\it point scheme} and {\it line scheme}, respectively. The
point scheme first appeared in \cite{ATV1,ATV2}, 
by  Artin-Tate-Van den Bergh, and the line scheme first appeared in  
\cite{SV1,SV2}, by Shelton and Vancliff. 

The spaces $\PP_{nc}$ are usually defined in terms of an algebra that plays the 
role of a homogeneous coordinate ring. If $\PP_{nc}$ has a homogeneous coordinate 
ring that is one of the %3-dimensional 
quadratic Artin-Schelter regular algebras of global dimension three
discussed in \S\ref{sect.ATV}, 
then the point scheme for $\PP_{nc}$ is the space $\Gamma$ defined in 
\S\ref{sect.ATV}, which is, up to isomorphism, the scheme ${\bf Y}_R$ 
in \S\ref{sect.YR}; moreover, the line scheme for those algebras is $\PP^2$.

\subsubsection{The line scheme in greater generality}
\label{sect.line.scheme.general}
Let $V$ be a 4-dimensional vector space, $R$ a 6-dimensional subspace of $V
\otimes V$, and $A=TV/(R)$. Suppose further (and the uninitiated reader can 
ignore this sentence which is only here for reasons of intellectual honesty) 
that $A$~is 
%a 4-dimensional 
an Artin-Schelter regular $\Bbbk$-algebra (see \cite[Defn. 8.1]{StVdB01}, for example)
and a domain of global dimension four with Hilbert series $(1-t)^{-4}$; crudely, $A$~is a good non-commutative analogue of the 
polynomial ring on four variables.\footnote{Artin-Tate-Van den Bergh showed that  noetherian Artin-Schelter regular algebra of global dimension 4 and Hilbert series $(1-t)^{-4}$ is a domain \cite[Thm. 3.9]{ATV2}.} 

The conditions on $A$ are a little stronger than Conditions 2.1-2.3 and 2.8 in Shelton and Vancliff's paper \cite{SV1}
where it is shown that under those hypotheses the line scheme for~$A$ (or, rather, for the non-commutative analogue 
of~$\PP^3$ that has $A$~as a homogeneous coordinate ring) is isomorphic to 
${\mathbf X}_R$. (In \cite{SV1}, ${\mathbf X}_R$ is denoted by $\cL^\perp$; see  \cite[p.585]{SV1}.) 

The point scheme parametrizes isomorphism classes of point modules and the line scheme parametrizes
 isomorphism classes of   line modules. 
 Point (resp., line) modules are, by definition, the cyclic graded right $A$-modules having Hilbert series
 $(1-t)^{-1}$ (resp., $(1-t)^{-2}$). 

If $A=SV$, the symmetric algebra on a 4-dimensional vector space $V$, then the line modules are the modules of the form $A/A\ell^\perp$
where $\ell$ runs over the lines in $\Proj(A)=\PP(V^*) \cong \PP^3$ and $\ell^\perp$ is the subspace of $V$ vanishing on $\ell$. 
When $A$ has the properties in the first paragraph of \S\ref{sect.line.scheme.general}, a right $A$-module $M$ is a line module if  
and only if it is of the form $A/LA$ where $L$ is a 2-dimensional subspace of $V$ such that 
$L \otimes V \cap R \ne \{0\}$ \cite[Prop. 2.8]{LS93}; i.e., if and only if $L \in ({\bf X}_R)_{\rm red}$. 
Notice too that ${\mathbf X}_R$ is isomorphic to the subscheme  of $\PP(R)\subseteq \PP(V\otimes V)$  consisting of the tensors
that have rank $\le 2$.
%
%Such examples appear in non-commutative
%algebraic geometry where the focus is on finding non-commutative analogues of $\PP^n$
%or, equivalently, finding non-commutative algebras that have many properties in
%common with the polynomial ring.  We address these connections in~\S\ref{sect.motive}.
%

A basic problem in NCAG is to 
describe the line scheme of algebras like $TV/(R)$, especially when that line scheme has dimension one because that is believed to be the generic case. 
The accumulating evidence has suggested that the degree of the line scheme (when it is viewed as a subscheme of the ambient Pl\"ucker~$\PP^5$) is 20 if its 
dimension is  one (\cite{CV15,CS2,TV16}).  
%So an open problem for several years has been to determine whether its degree is always 20 when its dimension is 1.
 
Our main result settles the matter. 
Theorem~\ref{thm.main} states, in part, that if ${\mathbf X}_R$~has minimal dimension (namely, one), then it has degree 20 
 as a subscheme of the ambient Pl\"ucker~$\PP^5$.  
  
\subsubsection{}
A bewildering variety of 
%4-dimensional 
Artin-Schelter regular algebras of global dimension four with Hilbert 
series $(1-t)^{-4}$ have appeared over the past 20~years.
Perhaps their line schemes will exhibit a similar bewildering variety. If not, their 
geometric similarities will suggest hidden algebraic similarities.

\subsection{}
The contents of the paper are as follows. 

In \S\ref{prelim} we introduce some notation and state 
the main result in Theorem~\ref{thm.main}. In \S\ref{proof} we almost prove the main theorem---the first example in \S\ref{egs} is needed to
complete its proof.  In~\S\ref{egs}, we present four examples, the first of which
completes the proof of Theorem~\ref{thm.main} and Corollary~\ref{thm.main.2}. 
The examples in~\S\ref{egs} illustrate how widely the structure 
of~${\mathbf X}_R$ can vary as $R$~varies.
For example, the first~${\mathbf X}_R$ in~\S\ref{egs} is reduced, but the second
is not.  The examples in \S\S\ref{Skly.example} and \ref{exotic.Skly.example} involve 
elliptic curves: in the former, ${\mathbf X}_R$ is a $\PP^1$-bundle over an elliptic curve;
in the latter, it is a curve having seven irreducible components, three of which are 
elliptic curves, and four of which are smooth conics.   The example in
\S\ref{Skly.example} appears in the literature in the context of quantum statistical physics
via solutions to the quantum Yang-Baxter equation (\cite{OF89,Skl82,Skl83}).
We end in \S\ref{sect.motive} by saying a little  more about the connection with non-commutative algebraic geometry.

%%%%%%%%%%%%%%%%%%%%%%%%%%%%%%%%%%%%%%%%%%%%%%%%%%%%%%%%%%%%%%%%%

\subsection*{Acknowledgements}
\label{sec:acknowledgements}
We would like to thank an anonymous referee for useful comments on an earlier draft of the paper. 

A.C. was partially supported by NSF grant DMS-1565226.

M.V. was partially supported by NSF grant DMS-1302050.

\bigskip
\medskip

%%%%%%%%%%%%%%%%%%%%%%%%%%%%%%%%%%%%%%%%%%%%%%%%%%%%%%%%%%%%%%%%%
%  start section 2
%%%%%%%%%%%%%%%%%%%%%%%%%%%%%%%%%%%%%%%%%%%%%%%%%%%%%%%%%%%%%%%%%

\section{Preliminaries}
\label{prelim}

In this section we state the main theorem and establish notation that will be used 
in its proof.
Throughout the paper, $\Bbbk$ denotes an algebraically closed field. For the rest of the paper,   $V$ denotes a 4-dimensional $\Bbbk$-vector space. 

\subsection{}
We identify the space of 
$4\times 4$ matrices over $\Bbbk$ with $V \otimes V^*$ but, since we never use the 
multiplicative structure of the matrix ring, we often replace $V^*$ by $V$.
We write ${\sf G}(2,V)$ for the Grassmannian of 2-planes in~$V$, and 
${\sf G}(8,V^{\otimes 2})$ for the Grassmannian of 8-dimensional subspaces of  
$V \otimes V$, and we identify ${\sf  G}(2,V)$ with its image in ${\sf G}(8,V^{\otimes 2})$
under the closed immersion ${\sf  G}(2,V) \longrightarrow {\sf G}(8,V^{\otimes 2})$, 
where $Q \mapsto Q \otimes V$. 

\subsection{}
Let $R$ denote a 6-dimensional subspace of $V \otimes V$.  The set
$$
{\bf S}_R \; :=\; \{W \in {\sf G}(8,V^{\otimes 2}) \; | \; W \cap R \ne \{0\}\,\}
$$
 is a closed subvariety of $ {\sf G}(8,V^{\otimes 2})$; indeed, it is a special Schubert 
 variety by~\cite[p.~146]{F-YT}.
The geometric invariant that we associate to~$R$ is the scheme
$$
{\mathbf X}_R \; :=\; {\bf S}_R \cap {\sf  G}(2,V),
$$
where the scheme-theoretic intersection is taken inside $ {\sf G}(8,V^{\otimes 2})$.  
The reduced variety of $\mathbf X_R$ is
$$
({\mathbf X}_R)_{\rm red} \; :=  \;  \{Q  \in {\sf G}(2,V) \; | \;  (Q \otimes V) \cap R 
\ne \{0\}\,\} \; \subseteq \; {\sf G}(2,V).
$$

\subsection{}
Our main result is

\begin{theorem}
[\Cref{pr.U_flat,thm.main.2}]
\label{thm.main}
Let $R \in \mathsf G(6, V^{\otimes 2})$. 
\begin{enumerate}
  \item 
  Every irreducible component of ${\mathbf X}_R$ has dimension $\ge 1$
  \item 
   $\{\mathbf X_R \ | \ R \in \mathsf G(6, V^{\otimes 2}), \  \dim(\mathbf X_R) = 1 \}$ is a flat family.
  \item 
  If char$(\Bbbk) \neq 2$ and if  $\dim({\mathbf X}_R)=1$, then 
$\deg({\mathbf X}_R) =20$, where  $\mathbf X_R$ is viewed as a subscheme of $\PP^5$ via
the Pl\"ucker embedding ${\sf G}(2,V) \to \PP(\bwedge^2V) \cong \PP^5$ 
$($cf.~\S\ref{P.emb}$)$.
%${\sf G}(2,V) \to \PP^5$ of\/ \S\ref{P.emb}.
\end{enumerate}
\end{theorem}

\bigskip
 
%%%%%%%%%%%%%%%%%%%%%%%%%%%%%%%%%%%%%%%%%%%%%%%%%%%%%%%%%%%%%%%%%
%  start section 3
%%%%%%%%%%%%%%%%%%%%%%%%%%%%%%%%%%%%%%%%%%%%%%%%%%%%%%%%%%%%%%%%%

\section{The proof of \Cref{thm.main}}
\label{proof}

In this section, we complete most of the proof of Theorem~\ref{thm.main}.
Our approach is to consider ${\mathbf X}_R$ as a member of a flat family of schemes and to show that the degree of each fiber is constant across the whole family. The example in \S\ref{gsca} exhibits a fiber having degree 20 so all fibers must have degree twenty.

\subsection{}
We will examine various families over the base scheme ${\sf G}(6,V^{\otimes 2})$. 
 Define 
\begin{align*}
{\bf G}&\; := \; {\sf G}(6,V^{\otimes 2})\times {\sf G}(8,V^{\otimes 2}),
\\
{\bf G}(1,3)&\; := \; {\sf G}(6,V^{\otimes 2})\times {\sf G}(2,V),    \quad \hbox{which we identify with its image in ${\bf G}$ under the closed} 
\\
&   \phantom{     \; := \; {\sf G}(6,V^{\otimes 2})\times {\sf G}(2,V), x \quad}  \hbox{immersion $(W,Q) \mapsto (W,Q \otimes V)$,}
\\ 
{\bf S}_R & \; := \; \{W \in {\sf G}(8,V^{\otimes 2}) \; | \; W \cap R \ne \{0\}\,\},
 \\ 
{\bf S} & \; := \;  \{(R,x)   \; | \; R \in  {\sf G}(6,V^{\otimes 2}), \, x \in {\mathbf S}_R\} \;  \subseteq  \;  {\bf G},
\\ 
  {\mathbf X}_R  & \; := \;  {\bf S}_R \cap {\sf G}(2,V), \quad \hbox{the scheme-theoretic intersection 
taken inside ${\sf G}(8,V^{\otimes 2})$}, 
 \\
 {\mathbf X} & \; := \; {\bf S} \cap {\bf G}(1,3), \quad \hbox{the scheme-theoretic intersection taken inside ${\bf G}(1,3) \subseteq   {\bf G}$},
\\
&\;  \phantom{:} =\; \{(R,x) \; | \; R \in  {\sf G}(6,V^{\otimes 2}), \, x \in {\mathbf X}_R\} \;  \subseteq  \;  {\bf G}(1,3) \subseteq    {\bf G} ,
\\ 
U& \; := \; \{R \in {\sf G}(6,V^{\otimes 2}) \; | \; \dim({\mathbf X}_R )=1\} \; \subseteq \; {\sf G}(6,V^{\otimes 2}),
\\ 
{\mathbf X}_U  & \; := \; \{(R,x)   \; | \; \dim({\mathbf X}_R)=1, \, x \in {\mathbf
X}_R\} \; \subseteq \; {\mathbf X}. 
\end{align*}
and  
\begin{align*}
\pi &: {\bf G} \longrightarrow {\sf G}(6,V^{\otimes 2}), \quad \hbox{the projection,}
\\ 
f& : {\mathbf X} \longrightarrow  {\sf G}(6,V^{\otimes 2}), \quad \hbox{the
restriction of  $ \pi$ to ${\mathbf X}$}.  
\end{align*}
Thus ${\mathbf X}_U =f^{-1}(U)=\pi^{-1}(U) \cap {\mathbf X}$.  
In general, we denote the restriction of a family to a subscheme $U\subseteq {\sf G}(6,V^{\otimes 2})$ by the subscript $U$.

% % \begin{remark}\label{re.kleiman}
% %   In characteristic zero ${\mathbf X}_R$ is smooth for an open dense set of relation spaces $R\in {\sf G}(6,V^{\otimes 2})$.
% % 
% %   Indeed, the variety ${\sf G}(8,V^{\otimes 2})$ where the intersection ${\mathbf X}_R={\mathbf S}_R\cap {\sf G}(2,V)$ is taken is a homogeneous space under the general linear group $G=GL(V\otimes V)$. Now fix an arbitrary $R_0\in {\sf G}(6,V^{\otimes 2})$. It follows from \cite[Theorem III.3.10]{Hart} that for $\sigma\in G$ ranging over a dense open subset, the intersection
% %   \begin{equation*}
% %     {\sf X}_{R_0^{\sigma}}={\mathbf S}_{R_0^{\sigma}}\cap {\sf G}(2,V)
% %   \end{equation*}
% %   is smooth. 
% % \end{remark}
% %  

\begin{proposition}\label{pr.U_open}
  The set $U$ is a dense open subset of ${\sf G}(6,V^{\otimes 2})$.
\end{proposition}
\begin{proof}
Each ${\bf S}_R$ is a Schubert variety and, by~\cite[pp.~193-196]{GH78}, for
example, its codimension in ${\sf G}(8,V^{\otimes 2})$ is three. It follows that 
the codimension of ${\bf S}$ in ${\bf G}$ is three. Since $\dim({\sf G}(2,V))=4$, the 
codimension of ${\bf G}(1,3)$ in ${\bf G}$ is
$\dim({\sf G}(8,V^{\otimes 2}))-4$. It follows that every irreducible 
component~${\mathbf X}_i$, $i\in I$, 
of ${\bf S} \cap {\bf G}(1,3)={\mathbf X}$ has codimension 
$\le \dim({\sf G}(8,V^{\otimes 2}))-1$. 
In other words, the relative dimension, 
$\dim({\mathbf X}_i)-\dim({\sf G}(6,V^{\otimes 2}))$, is at least $1$ for each $i\in I$. 

We will now apply~\cite[Exercise II.3.22(d)]{Hart} to the families obtained by
restricting the projection $\pi$ to the reduced
 subschemes  $({\mathbf X}_i)_{\rm red}$ (and co-restricting to the scheme-theoretic images of these
restrictions). The conclusion of the cited exercise is that ${\bf Y}:=\{x\in {\mathbf
X} \; | \; \dim({\mathbf X}_{\pi(x)}) \ge 2\}$
is a closed subscheme of ${\mathbf X}$. Since $\pi|_{\mathbf X}$ is a projective morphism, it is 
closed.  
Since $U = {\sf G}(6,V^{\otimes 2})- \pi({\bf Y})$, we find $U$~is open.
By~\cite{SV2}, $U$~is nonempty, and therefore dense.
\end{proof}

The following result is embedded in the proof of \Cref{pr.U_open}.

\begin{proposition}
Let $R \in {\sf G}(6,V^{\otimes 2})$. If $\dim({\mathbf X}_R)=1$, then every
irreducible component of ${\mathbf X}_R$ has dimension 1.
\end{proposition}

\begin{proposition}
\label{prop.flat.family}
Let $G$ be an algebraic group acting on $\PP^n$ and let $T$~be an integral noetherian
scheme endowed with a transitive action of $G$. Suppose $G$ acts diagonally on $T \times
\PP^n=\PP^n_T$. If  ${\mathbf X} \subseteq \PP^n_{T}$ is a $G$-stable 
closed subscheme, then ${\mathbf X}$ is flat over $T$.
\end{proposition}
\begin{proof}
Let $t,t'\in T$. There exists $g \in G$ such that $t'=g(t)$. The action of $g$ is such that the diagram
$$
\xymatrix{
{\mathbf X}_t \ar[d]_g \ar[r] & \PP_{\Bbbk(t)}^n \ar[d]^g
\\
{\mathbf X}_{t'}  \ar[r] & \PP_{\Bbbk(t')}^n 
}
$$
commutes. It follows that ${\mathbf X}_t$ and ${\mathbf X}_{t'}$ have the same Hilbert polynomial. 
It now follows from~\cite[Thm. III.9.9]{Hart} that ${\mathbf X}$~is flat over~$T$. 
\end{proof}

\begin{corollary}
\label{cor.S.flat}
The scheme $\bf S$ is flat over ${\sf G}(6,V^{\otimes 2})$.
\end{corollary}
\begin{proof}
Let  ${\sf G}(8,V^{\otimes 2}) \to \PP(\bwedge^8(V\otimes V))$ be the Pl\"ucker embedding 
and consider $\bf S$ as a closed subscheme of ${\sf G}(6,V^{\otimes 2}) \times  
\PP(\bwedge^8(V\otimes V))$. 
Let $\GL(V\otimes V)$ act diagonally on the previous product in the 
obvious way. It is clear that the closed subscheme $\bf S$ is stable under the action of 
$\GL(V\otimes V)$ and that the action of $\GL(V\otimes V)$ on ${\sf G}(6,V^{\otimes 2})$ 
is transitive. The result now follows from \Cref{prop.flat.family}.
\end{proof}

\begin{lemma}
Both ${\mathbf X}_U $ and ${\bf S}$ are Cohen-Macaulay schemes.
\end{lemma}
\begin{proof}
Let $\pi:{\bf S} \to {\sf G}(6,V^{\otimes 2})$ be the restriction of the projection. 
In~\cite{Hoch73}, 
Hochster proved that Schubert varieties in Grassmannians are  Cohen-Macaulay.
Thus,
${\bf S}_R$ is Cohen-Macaulay for all $R \in {\sf G}(6,V^{\otimes 2})$.  Let $x \in {\bf S}$.
By \Cref{cor.S.flat}, the natural map $u:\cO_{\pi(x),{\sf G}(6,V^{\otimes 2})} \to \cO_{x,{\bf S}}$  is a
flat homomorphism of noetherian local rings.
 The closed fiber of~$u$, which is $\cO_{x,{\bf S}_{\pi(x)}}$, is Cohen-Macaulay. Since $\cO_{\pi(x),{\sf G}(6,V^{\otimes 2})}$ is 
 also Cohen-Macaulay,  \cite[Cor.~23.3]{Mats89}~implies that $\cO_{x,{\bf S}}$~is Cohen-Macaulay. 
Thus ${\bf S}$ is Cohen-Macaulay.

The Cohen-Macaulay property for ${\mathbf X}_U $ will follow from \cite[Prop.~18.13]{Ebud95} 
applied to the following setup.

Recall that ${\mathbf X}_U = {\bf S}_U\cap {\bf G}(1,3)_U$. 
Let $A=\cO_{x,{\mathbf S}_U }$ be the local ring at a point $x\in {\mathbf X}_U $. 
Since ${\bf G}(1,3)_U$ is regular and hence a local complete intersection, we can find 
$n=\dim({\sf G}(8,V^{\otimes 2}))-4$ generators for the ideal~$I$ of~$A$ such that 
$\cO_{x,{\bf X}_U}=A/I$. 

%Let $A=\cO_{x,{\mathbf X}_U }$ be the local ring at a point $x\in {\mathbf X}_U $. Recall
%that ${\mathbf X}_U $ is the intersection ${\bf S}_U\cap {\bf G}(1,3)_U$. Since ${\bf
%G}(1,3)_U$ is regular and hence a local complete intersection, we can find 
%$n=\dim({\sf G}(8,V^{\otimes 2}))-4$ generators for the ideal~$I$ of~$A$ such that 
%$\cO_{x,{\bf S}_U}=A/I$. 

Since we are restricting our families to~$U$, the codimension of~$I$ in~$A$ is
precisely~$n$, but then the hypotheses of~\cite[Prop.~18.13]{Ebud95} are met, and so the 
ring~$A/I$ is Cohen-Macaulay.  
\end{proof}

\begin{theorem}
\label{pr.U_flat}
The scheme ${\mathbf X}_U $ is flat over $U$. 
\end{theorem}
\begin{proof}  
Let $x \in {\mathbf X}_U = f^{-1}(U)$. 
Let $B= \cO_{f(x),U}$ and $A=\cO_{x,{\mathbf X}_U }$. We must show that $A$ is a flat $B$-module.
Let $\fp$~denote the maximal ideal of~$B$.

Because ${\mathbf X}_U $ is Cohen-Macaulay,  $A$ is  a Cohen-Macaulay ring. 
Since $\dim(({\mathbf X}_U)_{\pi(x)})=1$, $\Kdim(A) = \Kdim(B)+\Kdim(A/A\fp)$.
Taken together, this equality, the fact that $A$ is Cohen-Macaulay, and 
\cite[Thm.~18.16(b)]{Ebud95}, imply that $A$ is a flat $B$-module.  
\end{proof}

\begin{corollary}
\label{thm.main.2}
Suppose char$(\Bbbk) \neq 2$.
If $\dim({\mathbf X}_R)=1$, then $\deg({\mathbf X}_R) = 20$, where $\mathbf X_R$ is viewed as
a subscheme of $\PP^5$ via the Pl\"ucker embedding
${\sf G}(2,V) \to \PP(\bwedge^2V) \cong \PP^5$.
\end{corollary}
\begin{proof}
Since  ${\mathbf X}_U $ is flat over $U$ and $\Bbbk$ is algebraically closed, 
the Hilbert polynomial of $X_R$, viewed as a closed subscheme of $\PP^5$,
is the same for all $R \in U$ \cite[Thm.~III.9.9]{Hart}. Since $\dim({\mathbf X}_R)=1$, the leading coefficient of its Hilbert polynomial is the 
degree of ${\mathbf X}_R$ as a closed subscheme of $\PP^5$. Thus, $\deg({\mathbf X}_R)$ is the same for all $R \in U$.
We therefore need only exhibit a single~$R$ for which $\deg({\mathbf X}_R) =20$; this is 
done in \S\ref{gsca} below.
\end{proof}

Modulo the example needed to finish the proof of  Corollary~\ref{thm.main.2}, 
the proof of Theorem~\ref{thm.main} is now complete.

\bigskip

%%%%%%%%%%%%%%%%%%%%%%%%%%%%%%%%%%%%%%%%%%%%%%%%%%%%%%%%%%%%%%%%%
%  start section 4
%%%%%%%%%%%%%%%%%%%%%%%%%%%%%%%%%%%%%%%%%%%%%%%%%%%%%%%%%%%%%%%%%

\section{Examples}
\label{egs}

This section exhibits various examples of ${\mathbf X}_R$,  
three of which have dimension one. 
Unless otherwise stated, we suppose char$(\Bbbk) \neq 2$ in this section.

Our first example, in \S\ref{gsca} below, requires some Pl\"ucker-coordinate
calculations so we first introduce the notation we need for that discussion.

\subsection{Two Pl\"ucker embeddings}
\label{P.emb}
Let $\bwedge^2V$ denote the second exterior power of $V$.
Fix an ordered basis $\{x_1, \ldots , x_4\}$ for~$V$.
The image of  a point $Q= \Bbbk v_1 \oplus \Bbbk v_2$ in $\mathsf G(2,V)$ 
under the Pl\"ucker embedding is the point $(N_{12},N_{13},N_{14},N_{23},N_{24},N_{34}) \in \PP(\bwedge^2V)$  where $N_{jk}$ is the coefficient of 
$x_j \wedge x_k$ in $v_1 \wedge v_2$.
If $v_i = \sum_{j=1}^4 \l_{ij} x_j$, $i=1,2$, then  
\[
v_1 \wedge v_2 = 
\sum_{1 \leq j< k \leq 4}
 (\l_{1j}\, \l_{2 k} - 
        \l_{1 k}\, \l_{2 j}) \, x_j \wedge x_k
\]
so $N_{j k} = \l_{1j}\, \l_{2 k} - \l_{1 k}\, \l_{2 j}$. 
Hence, the $N_{jk}$ are the $2 \times 2$ minors of the matrix 
\[
N = 
\begin{pmatrix}
\l_{11} &  \l_{12} &  \l_{13} &  \l_{14} \\  
\l_{21} &  \l_{22} &  \l_{23} &  \l_{24}   
\end{pmatrix}.
\]
The image of~$Q$ in $ \PP(\bwedge^2V)$ does not depend on the choice of basis 
for~$Q$ because changing that basis corresponds to performing row operations 
on~$N$ and this corresponds to multiplying $v_1 \wedge v_2$ by the determinant 
of the corresponding change-of-basis matrix.

%\label{perp}
Since $\dim_\Bbbk(V)=4$, the map ${\sf G}(2,V) \to {\sf G}(2,V^*)$ given by
$$
Q \mapsto Q^\perp \; :=\; \{\a \in V^* \; | \; \a(Q)=0\}
$$
is an isomorphism. 
To make this explicit, we use the Pl\"ucker embedding ${\sf G}(2,V^*) \to \PP(\bwedge^2V^*)$ obtained by 
taking the basis for  $V^*$ that is dual to $\{x_1, \ldots, x_4\}$. The image of $Q^\perp$ in $ \PP(\bwedge^2V^*)$
is then the point $(M_{12}, M_{13}, M_{14}, M_{23}, M_{24}, M_{34})$ where   
\[ 
M_{ij} \; = \;  (-1)^{i+j+1}N_{k\ell} = (-1)^{k+\ell+1}N_{k\ell}
\]
whenever $i<j$, $k < \ell$, and  $\{i, j, k,\ell\} = \{1, 2, 3, 4\}$.

We write $\mathbf X_R^\perp$ for the image of ${\bf X}_R$ under the isomorphism ${\sf G}(2,V) \to {\sf G}(2,V^*)$.   
It is clear that $\deg(\mathbf X_R) = \deg(\mathbf X_R^\perp)$.

\subsubsection{}
\label{lscheme}
The scheme $\cL$ mentioned in the abstract is the scheme $\mathbf X_R^\perp$. 
 In that context, $\mathbf X_R^\perp$ is the natural 
scheme to consider, as it describes which lines in~$\PP^3$ correspond to 
line modules over the quadratic algebra.

\subsection{An example with $\mathbf X_R$ reduced and $\dim(\mathbf X_R) = 1$}
\label{gsca}
Let $\g \in \Bbbk^\times$, let $V = \bigoplus_{i=1}^4 \Bbbk x_i$, and let $R$~denote the linear span of the six elements 
\[
\begin{array}{lll}
x_4 \otimes x_1 - i x_1 \otimes x_4, \qquad & x_3\otimes x_3 - x_1\otimes x_1, \qquad &
x_3 \otimes x_1 - x_1 \otimes x_3 + x_2\otimes x_2,\\[2mm]
x_3 \otimes x_2 - i x_2 \otimes x_3, & x_4\otimes x_4 - x_2\otimes x_2, & 
x_4 \otimes x_2 - x_2 \otimes x_4 + \g x_1\otimes x_1,
\end{array}
\]
where $i\in \Bbbk$, $i^2 = -1$.  

We will show that $\dim(\mathbf X_R) = 1$ and $\deg({\bf X}_R)=20$.  
As explained in \S\ref{P.emb}, it suffices to prove these equalities for $\mathbf X_R^\perp$ and that is what we will do.

Let $S=\Bbbk[M_{12}, \ldots, M_{34}]$ denote the polynomial ring that is the
homogeneous coordinate ring of $\PP(\bwedge^2V^*)$. As explained 
in \cite[Appendix~A.2]{CV15}, $\mathbf X_R^\perp = \Proj(S/I)$
where $I$ is the ideal in~$S$ generated by the 46 polynomials listed in that appendix.
In \cite{CV15}, the degree of $\mathbf X_R^\perp$ is computed under the 
assumption that ${\rm char}(\Bbbk)=0$. Here we use a different method to 
compute its degree and only assume char$(\Bbbk) \neq 2$.

Let $h = M_{12} + M_{13} + i M_{14} + i M_{23} + M_{24}$ and  $H=\{h =0\} \subseteq \PP(\bwedge^2V^*)$.\footnote{The linear form $h$ is of no particular significance and is only chosen as it is to simplify the calculations.} 
Using the explicit description of the 46 polynomials generating $I$ that were mentioned above, and certain Gr\"obner basis calculations,
we will show that the scheme-theoretic intersection $H \cap \mathbf X_R^\perp$ has 
exactly 20~points, counted with multiplicity.
We can write $(H \cap \mathbf X_R^\perp)_{\rm red}=Y_a \cup Y_b \cup \cdots
\cup Y_f$ where $Y_a$ is the zero locus of the polynomials in (a) below, and so on:
\begin{enumerate}
\item[(a)] $M_{12}$, \ $M_{13}$, \ $M_{23}$, \ $M_{34} -1$, \ 
             $M_{14} - i M_{24}$, \ $M_{24}^3  - M_{24}^2 -1$;  \phantom{\Big)}
\item[(b)] $M_{13}$, \ $M_{24}$, \ $M_{23}-1$, \ $M_{12} + i M_{14}+i$, \ 
           $i(M_{14} + 1) M_{34} - M_{14}$,  \phantom{\Big)}\\ 
	   $2 M_{14}^4 + (6 - \g) M_{14}^3 + (9-2\g) M_{14}^2 + 
	   (6-\g)M_{14} + 2$; \phantom{\Big)}
\item[(c)] $M_{13}$, \ $M_{14}$, \ $M_{34}$, \ $M_{23} -1$, \ 
             $M_{12} + M_{24}+i$, \ 
	     $M_{24}^3 + 4i M_{24}^2 -4 M_{24} - 2i$; \phantom{\Big)}
\item[(d)] $M_{14}$, \ $M_{23}$, \ $M_{13} -1$, \ $M_{12} + M_{24}+1$, \ 
           $(M_{24} + 1) M_{34} + M_{24}$, \ $(M_{24} + 1)^4  + M_{24}^2$;  \phantom{\Big)}
\item[(e)] $M_{12}$, \ $M_{14}$, \ $M_{24}$, \ $M_{13} -1$, \ 
             $M_{23} - i$, \ $M_{34}^3  + M_{34} +\g$;  \phantom{\Big)}
\item[(f)] $M_{23}$, \ $M_{24}$, \ $M_{34}$, \ $M_{13} -1$, \ 
             $M_{12} + 1 + i M_{14}$, \ 
	     $2 M_{14}^3 - 4i M_{14}^2 + (\g - 3) M_{14} +i$.
\end{enumerate}

%Recall that char$(\Bbbk) \neq 2$. A straightforward, but tedious, computation with the polynomials in (a)-(f)  shows that  the 
%cardinality of $(H \cap \mathbf X_R^\perp)_{\rm red}$ is  at most $20$, and equal to 20 for generic~$\Bbbk$.  
%We will now show that the non-reduced scheme, 
%$H \cap \mathbf X_R^\perp$, has exactly 20 points when counted with multiplicity.

Let $J$ denote the ideal in $S$ generated by $I$ and $h$. 
Thus, $H \cap {\bf X}_R^\perp = \Proj(S/J)$.  
We first consider the subscheme of $H \cap {\bf X}_R^\perp$ supported on $Y_a$. 
Let $S_a'$ denote
the localization of~$S$ formed by inverting all homogeneous polynomials
that are nonzero at every point of~$Y_a$, and let $S_a$ denote the degree-0 component of the graded ring $S_a'$.
Let $J_a=JS_a' \cap S_a$. Since $M_{34}^{-1} \in S_a'$,  $S_a$ contains the
elements $m_{ij}:=M_{ij}M_{34}^{-1}$. 
For each $i$, $M_{i4}$ is nonzero on $Y_a$, so 
$m_{i4}^{-1} \in  S_a$. 
From the explicit list of polynomials in \cite[Appendix~A.2]{CV15} that generate~$I$, 
it follows that $J_a$~contains  
$$
i m_{23} m_{24} - m_{13}, \quad
i m_{23} m_{24} + m_{13} ,\quad \mbox{\rm and} \quad
m_{14}^2 - i m_{13} m_{14} + i m_{23} + m_{14} m_{23} -i m_{14} m_{24}.
$$
Since char$(\Bbbk) \neq 2$,  $J_a$ contains $m_{13}$ and $m_{23}$. 
It follows that $m_{14} - i m_{24} \in J_a$. Using the appendix of \cite{CV15} again, 
and the fact that $h\in J$,  
one sees that  $J_a$ is generated by $m_{24}^3 - m_{24}^2 - 1$.  Thus, $S_a/J_a = \Bbbk[m_{24}]/(m_{24}^3 - m_{24}^2 - 1)$.
In particular, $\dim_\Bbbk(S_a/J_a)=3$.

We now apply the same technique (and analogous notation) to case~(b).  
Since $M_{23}$ and $M_{14}$ are nonzero on $Y_b$, $m_{23}$ and $m_{14}$ are units in $S_b$. 
After consulting \cite[Appendix~A.2]{CV15} again, one sees that $J_b$ contains $m_{24}$ and $m_{13}$, and then that 
$m_{14} -i (m_{14} + 1)m_{34} \in J_b$.  Eventually one sees that $J_b$  is generated by 
$2 m_{14}^4 + (6 - \g) m_{14}^3 + (9-2\g) m_{14}^2 + (6-\g)m_{14} + 2$ and hence that $S_b/J_b$ is $\Bbbk[m_{14}]$ modulo the ideal generated by this element.
In particular, since char$(\Bbbk) \ne 2$, $\dim_\Bbbk(S_b/J_b)=4$. 

Similar arguments address the remaining cases (c)-(f). We omit the details.
Using \cite[\S3.2.4]{Chandler} one sees that the other four rings $S_c/J_c, \ldots,S_f/J_f$ are polynomial rings in one variable, $x$ say, modulo the ideals generated by
the elements $x^3 + 4ix^2 -4x -2i$,  $(1+x)^4 + x^2$,  $x^3 + x +\g$, and  $2 x^3 - 4i x^2 + (\g-3)x +i$,
respectively. %These rings have  dimensions 3, 4, 3, and 3, respectively.  
Therefore
$$
\dim_\Bbbk(S_a/J_a) + \cdots + \dim_\Bbbk(S_f/J_f) \;=\; 3+4+3+4+3+3 \;=\; 20,
$$
so $H \cap \mathbf X_R^\perp$ has exactly 20 points when counted with multiplicity.
Hence $\deg(\mathbf X_R^\perp)=20$ and, by \S\ref{P.emb}, 
$\deg({\mathbf X}_R) = 20$. The 
proofs of Theorem~\ref{thm.main} and Corollary~\ref{thm.main.2} are now complete.

Combining this result with discussions in~\cite{CV15} and \cite[Chapter~3]{Chandler}, it follows that in this example $\mathbf X_R^\perp$  and $\mathbf X_R$~are reduced for generic $\gamma\in \Bbbk^\times$ (provided char$(\Bbbk)\ne 2$). In fact, the results in \cite{CV15} show that if char$(\Bbbk) = 0$ and $\gamma$ is generic then ${\mathbf X}_R$~is isomorphic to the union in~$\PP^5$ of
\begin{itemize}
\item a degree-four nonplanar elliptic curve contained in a~$\PP^3\subset \PP^5$
\item four planar elliptic curves, and
\item two non-singular conics.  
\end{itemize}

\subsection{An example with $\mathbf X_R$ nonreduced and $\dim(\mathbf X_R) = 1$}
\label{1pt}
Suppose char$(\Bbbk) = 0$. Let $V = \bigoplus_{i=1}^4 \Bbbk x_i$ as in \S\ref{gsca} and let $R$~denote the linear span 
of the six elements 
\[
\begin{array}{lcl}
x_1 \otimes x_2 - i x_2 \otimes x_1 - x_4\otimes x_4, & 
\qquad & x_2 \otimes x_3 - i x_3 \otimes x_2, \\[2mm]
x_1 \otimes x_3 - i x_3 \otimes x_1 - x_2\otimes x_2, &        & 
x_3 \otimes x_4 - i x_4 \otimes x_3, \\[2mm]
x_1 \otimes x_4 - i x_4 \otimes x_1 - x_3\otimes x_3, &        & 
x_4 \otimes x_2 - i x_2 \otimes x_4,
\end{array}
\]
where $i \in \Bbbk$, $i^2 = -1$. 

This example is
discussed briefly in~\cite[\S2.4]{SV2}, where it is noted that the
reduced variety of $\mathbf X_R^\perp$ is the union of a ``triangle'' and a curve in 
$\PP^5$. In particular, $\dim(\mathbf X_R) = 1$.   
We will show here that $\mathbf X_R$ is the union of a 
multiplicity-three triangle 
(i.e., three intersecting lines, each of multiplicity three) and a
reduced irreducible curve. 
As in \S\ref{gsca}, we examine $\mathbf X_R^\perp$ and use the 
Pl\"ucker coordinates $M_{12}, \ldots,  M_{34}$ of \S\ref{P.emb}, 
and write $S = \Bbbk[M_{12},\dots, M_{34}]$.

The method used to produce the defining polynomials of~$\mathbf X_R^\perp$ 
is given in~\cite{SV2}. Intersecting $\mathbf X_R^\perp$ with the hyperplane 
given by the zero locus of~$M_{23}$ verifies that the reduced variety 
of~$\mathbf X_R^\perp$ contains the triangle of points given by 
$M_{12}M_{13}M_{14} = 0$ and $M_{23} = M_{24} = M_{34} = 0$. 

Let $H$ denote the hyperplane given by the zero locus of the polynomial 
$h = M_{12} + M_{13} + M_{14}$, and let $J$ denote the ideal in $S$
generated by~$h$ and the polynomials that define $\mathbf X_R^\perp$. 
We first consider the set $Y_a$ of points of $\mathbf X_R^\perp \cap H$  
at which $M_{23}$ is nonzero.  We adopt notation similar to that used 
in \S\ref{gsca}. 
Thus, let $S_a'$ denote the localization of~$S$ formed by inverting all 
homogeneous polynomials that are nonzero at every point of~$Y_a$, and 
let $S_a$ denote the degree-0 component of the graded ring $S_a'$,
and let $J_a=JS_a' \cap S_a$. 
Computing a Gr\"obner basis (with, e.g., Wolfram's Mathematica) 
of $J_a$ shows that $S_a/J_a$ is isomorphic to a polynomial ring in 
one variable with exactly one relation~$g$, where $\deg(g) = 11$ and 
$g$~has no repeated factors.  This shows that $\mathbf X_R^\perp \cap H$  
has exactly eleven distinct points, each of multiplicity one, that have 
nonzero $M_{23}$-coordinate.

On the other hand, $\mathbf X_R^\perp \cap H$ contains exactly three 
distinct points at which $M_{23}$ is zero.  These points lie on the 
aforementioned triangle and are
\[
p_1 = (0, 1, -1, 0, 0, 0), \qquad p_2 = (1, 0, -1, 0, 0, 0),\qquad 
p_3 = (1, -1, 0, 0, 0, 0).
\]
Since $R$ is invariant under the map that cyclically permutes $x_2$, 
$x_3$ and $x_4$, there is a corresponding map on the Pl\"ucker 
coordinates that 
cyclically permutes $p_1$, $p_2$ and $p_3$. Thus, the three points 
have the same multiplicity, and so it suffices to find the 
multiplicity of $p_1$.
Analogous to our argument in \S\ref{gsca}, 
let $S_b'$ denote the localization of~$S$ formed by inverting all 
homogeneous polynomials that are nonzero on $p_1$, and let $S_b$ denote 
the degree-0 component of the graded ring $S_b'$, and let 
$J_b=JS_b' \cap S_b$. 
A computation of a Gr\"obner basis (with, e.g., Wolfram's Mathematica) 
yields an element of the form $(m_{13}-1)^3 f \in J_b$, where 
$m_{13} := -M_{13}M_{14}^{-1} \in S_b$ and $f$~is the image in $S_b$ of a 
polynomial that is nonzero on~$p_1$. Hence, $f$~is invertible in $S_b$ 
and so $(m_{13}-1)^3 \in J_b$. Using this data to 
compute another Gr\"obner basis for $J_b$ yields that $S_b/J_b \cong
\Bbbk[x]/((x-1)^3)$, and so $\dim(S_b/J_b) = 3$, 
and hence $p_1$~has multiplicity three. Therefore, as discussed above, 
$p_2$ and $p_3$ also each have multiplicity three.

Consequently, $\mathbf X_R^\perp$ (and hence $\mathbf X_R$) is not a reduced scheme.

In fact, by intersecting $\mathbf X_R^\perp$ with other hyperplanes (e.g., the zero 
locus of $M_{13} - M_{14}$), we find that each nonvertex point has
multiplicity three, and each vertex of the triangle has multiplicity
seven. It follows that this part of $\mathbf X_R^\perp$ is a triple triangle.

\subsection{An example with ${\mathbf X}_R$ a surface}
\label{Skly.example}
Let $\{\a_1,\a_2,\a_3\} \subseteq \Bbbk-\{0,\pm 1\}$ and suppose that $\a_1+\a_2+\a_3+\a_1\a_2\a_3=0$. 
Let $\{x_0,x_1,x_2,x_3\}$ be a basis for a 4-dimensional vector space $V$.
Let $E$~be the quartic elliptic curve in~$\PP^3$ given by the intersection of any two 
of the quadrics
 \begin{equation}
 \label{eqns.for.E}
 \begin{cases}
  x_0^2+x_1^2+x_2^2+x_3^2 \; = \; 0,     
  \\
  x_0^2 - \a_2\a_3 x_1^2 - \a_3 x_2^2+\a_2 x_3^2  \; = \; 0,
\\
x_0^2+\a_3 x_1^2-\a_1\a_3 x_2^2-\a_1 x_3^2  \; = \; 0,
\\
x_0^2-\a_2 x_1^2+\a_1 x_2^2-\a_1\a_2 x_3^2  \; = \; 0.
  \end{cases}
  \end{equation}
  Every quartic elliptic curve in~$\PP^3$ is isomorphic to at least one of these~$E$'s.

 To save space, we write $x_i x_j$ for $x_i \otimes x_j \in V\otimes V$. 
Let $R$ be the linear span of the 6 elements
\begin{equation}
\label{4-dim-Sklyanin}
x_0x_i - x_i x_0 -\a_i(x_jx_k+x_kx_j)
\quad \hbox{and} \quad x_0x_i + x_i x_0 -\a_i(x_jx_k-x_kx_j)
\end{equation}
where $(i,j,k)$ runs over the cyclic permutations of $(1,2,3)$.

By \cite{LS93,St_lin},  ${\mathbf X}_R$~is isomorphic to a $\PP^1$-bundle over~$E$. More 
precisely, 
${\mathbf X}_R$~consists of the lines in $\PP^3$ whose scheme-theoretic intersection 
with~$E$ has multiplicity~$\ge 2$, and 
hence~$=2$. In other words, ${\mathbf X}_R$~parametrizes the set of secant lines to~$E$. 
It is isomorphic to the second symmetric power $S^2E=E \times E/\sim$ where $(x,y) \sim(y,x)$.

\subsection{An example involving elliptic curves with $\dim\big({\mathbf X}_R\big)=1$}
\label{exotic.Skly.example}
Retain the notation in \S\ref{Skly.example}, and let $R'$~be the  linear  span of the 
six~elements
\begin{equation}
\label{4-dim-exotic-Sklyanin}
x_0x_i - x_i x_0 -\a_i(x_jx_k-x_kx_j)
\quad \hbox{and} \quad x_0x_i + x_i x_0 -\a_i(x_jx_k+x_kx_j)
\end{equation}
where $(i,j,k)$ is a cyclic permutation of $(1,2,3)$. 
By~\cite{CS2},    
$$
{\mathbf X}_{R'} \; = \;  \big(C_1  \; \sqcup \; C_2 \; \sqcup \; C_3  \; \sqcup \;   C_4\big)  \; \cup \; 
\big(E_1 \; \sqcup \; E_2 \; \sqcup \;E_3 \big),
$$
where 
\begin{enumerate}
\item 
the $C_i$'s~are disjoint smooth plane conics, 
  \item 
  the $E_j$'s~are degree-4 elliptic curves that span different 3-planes in~$\PP^5$, and 
  \item 
each $E_j$~is isomorphic to~$E/(\xi_j)$ 
where $(\xi_1),(\xi_2),(\xi_3)$ are the three order-2 subgroups of~$E$, and
\item
$|C_i \cap E_j|=2$ for all $(i,j) \in \{1,2,3,4\} \times \{1,2,3\}$.
\end{enumerate}

\subsection{The origin of the examples in \S\S\ref{Skly.example} and~\ref{exotic.Skly.example}}

 There is a point $\tau \in E$ such that the space~$R$ given by 
(\ref{4-dim-Sklyanin})  has the following description (see~\cite{St_lin}, for example). 

Let $\cL$ be a line bundle of degree four on~$E$, and let $V=H^0(E,\cL)$. 
Then $R$~consists of those sections of the line bundle $\cL\boxtimes\cL$ on $E\times E$ whose divisor of zeros is of the form
\begin{equation*}
D+  \{(x,x+2\tau)\; |\;  x\in E\},
\end{equation*}
where $D$ is a divisor invariant under the involution $(x,y)\mapsto (y+2\tau,x-2\tau)$ such that 
\begin{equation*}
  \{(x,x-2\tau)\}
\end{equation*}
(the fixed-point set of the involution) occurs in $D$ with even multiplicity. 

The quotient $TV/(R)$ of the tensor algebra on $V$ modulo the ideal generated by~$R$ is 
a characteristic-free construction of the Sklyanin algebras (on four generators) 
of~\cite{OF89,Skl82,Skl83}, introduced by Sklyanin to study certain solutions to the 
Yang-Baxter equation. By~\cite{SS92}, the algebras $TV/(R)$~form a flat family of 
deformations of the symmetric algebra~$SV$, the polynomial ring on four variables,  with 
the data $(E,~\tau)$ acting as deformation parameters.

The algebras~$TV/(R')$ can be obtained from the Sklyanin algebras~$TV/(R)$ by means of
a cocycle twist construction, so they too have a geometric construction involving an
elliptic curve. We refer to~\cite{CS1,CS2} for details and the description of~${\mathbf X}_{R'}$. 

%\bigskip
%
%Further examples of the scheme ${\mathbf X}_R$ can be found 
%in~\cite{Chandler,CV15,SV2,TV16}.

\bigskip
\medskip

%%%%%%%%%%%%%%%%%%%%%%%%%%%%%%%%%%%%%%%%%%%%%%%%%%%%%%%%%%%%%%%%%
%  start section 5
%%%%%%%%%%%%%%%%%%%%%%%%%%%%%%%%%%%%%%%%%%%%%%%%%%%%%%%%%%%%%%%%%

\section{Motivation}
\label{sect.motive}
 
\subsection{Non-commutative algebraic geometry }
\label{ssect.nag}
Let $V$ be a 4-dimensional vector space. The symmetric algebra on $V$, $SV$, is a polynomial ring on four variables.
Thus $\Proj(SV) = \PP(V^*) \cong \PP^3$. In~\cite{FAC}, Serre proved that 
a certain quotient of the category of graded $SV$-modules is equivalent 
to the category, $\Qcoh(\PP^3)$, of quasi-coherent sheaves on~$\PP^3$. 

Since $SV=TV/({\rm Alt})$, where ${\rm Alt}$ is the subspace of $V \otimes V$
consisting of all skew-symmetric tensors, one might be led to perform the same quotient-category 
construction on other algebras $A=TV/(R)$ when $R$ is {\it any}\/ 6-dimensional subspace 
of $V\otimes V$ and, if $A$ is ``good'', meaning that it has the properties in the first paragraph of \S\ref{sect.line.scheme.general}, one might hope that this quotient category behaves
``like'' $\Qcoh(\PP^3)$. This hope is a reality in a surprising number of cases and has  
led to a rich subject
that goes by the name of non-commutative algebraic geometry --- 
see~\cite{MSRI2013,StVdB01} and the references therein.

From now on, we will assume that $A$ is ``good'' in the above sense. 

We will denote the appropriate quotient of the category of graded $A$-modules by ${\sf QGr}(A)$.
We will write $\Proj_{nc}(A)$ for the ``non-commutative analogue of $\PP^3$'' that has $A$ as its homogeneous coordinate
ring. This is a fictional object that is made manifest by declaring the category of quasi-coherent sheaves on $\Proj_{nc}(A)$
to be ${\sf QGr}(A)$.

\subsection{Points, lines, etc.}
The most elementary geometric features of $\PP^3$ are points, lines, planes, quadrics, and their incidence relations. 
There are non-commutative analogues of these and a first investigation of $\Proj_{nc}(A)$ involves finding its ``points''
and ``lines'' and the ``incidence relations'' among them. 

A non-commutative ring usually has far fewer two-sided ideals than a commutative
ring. Even when $A$ is ``good'',
it has far fewer two-sided ideals than $SV$. As a consequence, 
$\Proj_{nc}(A)$ usually has far fewer 
``points'' and ``lines'' than does~$\PP^3$. 

For example, $\Proj_{nc}(A)$ can have as few as twenty ``points'', which is the case 
for~$TV/(R)$ when $R$~is the space in~\S\ref{gsca} and $\g^2 \neq 4$ (\cite{CV15}), and is
the case for~$TV/(R')$ when $R'$~is the space in \S\ref{exotic.Skly.example} (\cite{CS1}),
or even just one point with multiplicity~20 as in~\S\ref{1pt} (\cite{SV99}). There
are many examples in the literature where the number of ``points'' is infinite, such
as the algebras in~\cite{S94} that are related to the Sklyanin algebra on four
generators (the latter being the algebras in~\S\ref{Skly.example}); 
moreover, the scheme~$\mathbf X_R$ of the generic member of the family of algebras 
in~\cite{S94} has dimension one.  Hence, 
the ``points'' in $\Proj_{nc}(A)$ form a scheme that is called the {\sf point scheme}. 
When $A$~is ``good'', the point scheme is a closed subscheme of~$\PP^3$; and if it 
has dimension zero, then its degree is~20 (cf.~\cite[p.~377]{VMSRI2013}). 

By~\cite{SV1,SV2}, there is also a  {\sf line scheme}  that classifies the
``lines'' in~$\Proj_{nc}(A)$. By~\cite{SV1},
the line scheme is isomorphic to ${\mathbf X}_R \subseteq {\sf G}(2,V)$ and always has dimension~$\ge 1$.  
Thus, the main result in this paper says that if $A$~is ``good'' and the line scheme has 
dimension 1, then its degree is 20.

\subsection{Rank-2 elements in $R$}
As before, let $R$~denote a 6-dimensional subspace of~$V \otimes V$.

The rank of an element $t \in V \otimes V$ is the smallest~$n$ such that 
$t=v_1 \otimes w_1 + \cdots + v_n \otimes w_n$ for some $v_i,w_i \in V$. The elements 
in~$V\otimes V$ of rank $\le r$ form a closed subvariety of $V \otimes V$. This subvariety
is a union of 1-dimensional subspaces, so the lines through 0 and the nonzero elements of 
rank $\le r$ form a closed subvariety $T_{\le r}$ of $\PP(V\otimes V)$. We write $T_r$ 
for $T_{\le r} - T_{\le r-1}$. Define ${\bf Z}_R$ to be the scheme-theoretic intersection
$$
{\bf Z}_R\; := \; T_2 \cap \PP(R),
$$
and define ${\bf Y}_R$ to be the scheme-theoretic intersection
$$
{\bf Y}_R \; := \; {\bf S}_R \cap {\sf G}'(2,V)
$$
in~${\sf G}(8,V^{\otimes 2})$, where ${\sf G}'(2,V)$~is the image of the closed immersion 
of~${\sf G}(2,V)$ in~${\sf G}(8,V^{\otimes 2})$ given by 
$Q \mapsto V \otimes Q$. Thus, ${\bf Y}_R={\mathbf X}_{\s(R)}$, where 
$\s:V\otimes V \to V\otimes V$ is the linear map $\s(u \otimes v)=v \otimes u$.

If  $t \in {\bf Z}_R$ and $t=a \otimes b+c \otimes d$, then the 2-dimensional subspace of $V$ spanned by $a$ and $b$ 
depends only on $t$ and not on its representation as  $a \otimes b+c \otimes d$. Thus, there is a well-defined morphism $\phi:{\bf Z}_R \to {\sf G}(2,V)$ given by 
$\phi(a \otimes b+c \otimes d) := \Bbbk a\oplus \Bbbk c$, and the image of~$\phi$ is~${\mathbf X}_R$. 
Similarly, there is a morphism  ${\bf Z}_R \to {\sf G}(2,V)$ given by
$a \otimes b+c \otimes d \mapsto \Bbbk b\oplus \Bbbk d$.  Thus, we have morphisms
$$
\xymatrix{
& {\bf Z}_R \ar[dl]_\phi \ar[dr] & 
\\
{\mathbf X}_{R} &&  {\bf Y}_{R} 
}
$$
When $TV/(R)$ is ``good'', the morphism~$\phi$ is an isomorphism; in particular, our
map~$\phi$ is the same as the morphism~$\phi$ in~\cite[Lemma 2.5]{SV1}, where it is 
proved to be an isomorphism (but our other notation differs from that in loc.\ cit.). 

\bigskip

It is suggested that the interested reader consult~\cite{MSRI2013,StVdB01} 
for further
details and references on studying non-commutative algebra via geometric techniques.

\bigskip

%%%%%%%%%%%%%%%%%%%%%%%%%%%%%%%%%%%%%%%%%%%%%%%%%%%%%%%%%%%%%%%%%
%  start bibliography
%%%%%%%%%%%%%%%%%%%%%%%%%%%%%%%%%%%%%%%%%%%%%%%%%%%%%%%%%%%%%%%%%

%\bibliography{line}
\bibliographystyle{plain}

\def\cprime{$'$} \def\cprime{$'$}

%\enlargethispage{11mm}
\bigskip
\bigskip

\end{document}